\documentclass[12pt,a4paper,reqno]{amsart} 
\usepackage{amssymb}
\usepackage{latexsym}
\usepackage{amsmath}
\usepackage{mathrsfs}
\usepackage{color}
\usepackage{calc}                         
\usepackage{cancel}

\newcommand{\scal}[2]{\langle #1,#2\rangle}

\newcommand{\rr}[1]{\mathbf R^{#1}}
\newcommand{\nn}[1]{\mathbf N^{#1}}

\newcommand{\zz}[1]{\mathbf Z^{#1}}

\newcommand{\nm}[2]{\Vert #1\Vert _{#2}}
\newcommand{\Nm}[2]{\left \Vert #1\right \Vert _{#2}}
\newcommand{\nmm}[1]{\Vert #1\Vert }

\newcommand{\ep}{\varepsilon}
\newcommand{\fy}{\varphi}

\newcommand{\cdo}{\, \cdot \, }

\newcommand{\supp}{\operatorname{supp}}

\newcommand{\vrum}{\vspace{0.1cm}}

\newcommand{\maclB}{\mathcal B}
\newcommand{\maclH}{\mathcal H}

\newcommand{\mascB}{\mathscr B}
\newcommand{\mascD}{\mathscr D}

\newcommand{\mascS}{\mathscr S}

\numberwithin{equation}{section}          

\newtheorem{thm}{Theorem}
\numberwithin{thm}{section}

\newtheorem{lemma}[thm]{Lemma}

\theoremstyle{definition}

\newtheorem{defn}[thm]{Definition}

\theoremstyle{remark}

\newtheorem{rem}[thm]{Remark}              

\title{Translation and modulation invariant Hilbert spaces}

\author{Joachim Toft}

\address{Department of Mathematics,
Linn{\ae}us University, V{\"a}xj{\"o}, Sweden}

\email{joachim.toft@lnu.se}

\author{Anupam Gumber}

\address{Department of Mathematics, Indian Institute
of Science, 560 012 Bangalore, India}
	
\email{anupamgumber@iisc.ac.in}

\author{Ramesh Manna}

\address{Department of Mathematics, Indian Institute of Science,
560 012 Bangalore, India}
	
\email{rameshmanna@iisc.ac.in}

\author{P. K. Ratnakumar}

\address{Harish-Chandra Research Institute (HBNI), Chhatnag Road,
Jhunsi, Allahabad, 211019, Uttarpradesh, India}

\email{ratnapk@hri.res.in}

\begin{document}

\begin{abstract}
We show that for any Hilbert space, $\maclH$, of distributions on $\rr d$
which is translation and modulation invariant, is equal to $L^2(\rr d)$,
with the same norm apart from a multiplicative constant.
\end{abstract}



\maketitle

\section{Introduction}\label{sec0}

\par

In the paper we show that any Hilbert space of distributions on $\rr d$ which is translation
and modulation invariant agrees with $L^2(\rr d)$. 
These considerations are strongly linked with Feichtinger's minimization property,
which shows that the modulation space (also called the
Feichtinger algebra) $M^{1,1}(\rr d)$ is the smallest non-trivial Banach space which is
norm invariant under translations and modulations. Our investigations may therefore
be considered as a Hilbert space analogy of those investigations which lead to
Feichtinger's minimization property.

\par

We remark that the search of the
smallest Banach space possessing such norm invariance properties, seems to be the
main reason that Feichtinger was led to introduce and investigate
$M^{1,1}(\rr d)$ and in its prolongation the foundation of classical modulation spaces
(see \cite{Fe1}). The space $M^{1,1}(\rr d)$ is small in the sense that it is contained in
any Lebesgue space $L^p(\rr d)$, as well as the Fourier image of these spaces (cf.
\cite{Fe1,Gc1,Toft20} and the references therein).
This fact is also an immediate consequence of Feichtinger's minimization property. On
the contrary, the modulation space $M^{\infty ,\infty}(\rr d)$, which is the dual of
$M^{1,1}(\rr d)$,
contains all these Lebesgue and Fourier Lebesgue spaces. By a straight-forward
duality approach it can be proved that for suitable assumptions on a translation
and modulation invariant Banach space $\mascB$, we have
$$
M^{1,1}(\rr d) \subseteq \mascB \subseteq M^{\infty ,\infty}(\rr d)
$$
(where the first inclusion is a reformulation of the Feichtinger's minimization property).

\par

Feichtinger's minimazation property has been extended in different ways, e.{\,}g. to
weighted spaces (see e.{\,}g. \cite[Chapter 12]{Gc1}), and to the quasi-Banach
situation (see e.{\,}g. \cite{Toft20}). At the same time minimization property has been
applied in e.{\,}g. non-uniform samplings, and for deducing sharp Schatten-von
Neumann and nuclear results for operators with kernels in modulation spaces
(see e.{\,}g. \cite{Toft20}).

\par

In our investigations we do not present any such weighted analogies in the
Hilbert space case.

\par

\section*{Acknowledgement}
The first author was supported by Vetenskapsr{\aa}det
(Swedish Science Council) within the project 2019-04890. He
is also grateful to Harish-Chandra
Research Institute, Prayagraj (Allahabad) for the excellent
hospitality and research facilities during his stay in the end of 2019.

\par

The second author is grateful for the support received from NBHM grant
(0204/19/2019R{\&}D-II/10472). She is also very grateful to
Harish-Chandra Research Institute, Prayagraj (Allahabad)
(HRI) for providing excellent research facilities and kind
hospitality during her academic visit to HRI.

\par

The third author is  thankful to Indian Institute of Science
(C.V. Raman PDF, file no: R(IA)CVR-PDF/2020/224 ) for
the financial support. He also thanks Harish-Chandra
Research Institute, Prayagraj (Allahabad) (HRI) for the visit and
excellent research facilities.

\par

%

\section{Translation and modulation invariant
Hilbert spaces}\label{sec2}

\par

In this section we first recall the definition of translation
and modulation invariant spaces. Thereafter we consider
such spaces which at the same time are Hilbert spaces
of distributions on $\rr d$. We show some
features on how differentiations and multiplications by polynomials
of such spaces behave in the inner product of such Hilbert
spaces. In the end we show that such Hilbert spaces agree
with $L^2(\rr d)$.

\par

We use the same notations as in \cite{Ho1}. The definition of
translation and modulation invariant quasi-Banach spaces is
given in the following.

\par

\begin{defn}\label{Def:InvariantSpaces}
Let $\maclB$ be a quasi-Banach space of measurable functions or (ultra-)distributions
on $\rr d$. Then $\maclB$ is called \emph{translation and modulation invariant},
if
$x\mapsto f(x-x_0)e^{i\scal x{\xi _0}}$ belongs to $\maclB$ and
$$
\nm {f(\cdo -x_0)e^{i\scal \cdo {\xi _0}}}{\maclB} = \nm f{\maclB}
$$
for every $f\in \maclB$ and $x_0,\xi _0 \in \rr d$.
\end{defn}

\par
\par

Our main result is as follows. 

\par

\begin{thm}\label{Thm:InvHilbertSpaces}
Let $\maclH$ be a translation and modulation invariant Hilbert space
on $\rr d$ which contains at least one element in $M^{1,1}(\rr d)\setminus 0$
and is continuously embedded in $\mascD '(\rr d)$. Then
$\maclH = L^2(\rr d)$ with
\begin{equation}\label{Eq:NormEquiv}
\nm f{\maclH} =  c\cdot  
\nm f{L^2(\rr d)},
\end{equation}
for some constant $c>0$ which is independent of $f\in \maclH = L^2(\rr d)$.
\end{thm}

\par

\begin{rem}
It is obvious that the constant $c$ in \eqref{Eq:NormEquiv}
can be evaluated by
$$
c=(\nm f{L^2(\rr d)} )^{-1}\nm f{\maclH}
$$
for any fixed $f\in \maclH \setminus 0$.
\end{rem}

\par

We need some preparations for the proof. Since
$\maclH$ in Theorem \ref{Thm:InvHilbertSpaces} is continuously embedded
in $\mascD '(\rr d)$, it follows by some straight-forward arguments that
$\maclH$ is continuously embedded in $\mascS '(\rr d)$. In order to be self-contained
we here present some motivations.

\par

In fact, let $Q_{d,r}$ be the cube $[0,r]^d\subseteq \rr d$, $\ep >0$ and
$0\le \fy \in C_0^\infty (\rr d)$ be such that
$$
\supp \fy \subseteq Q_{d,1+\ep}
\quad \text{and}\quad
\sum _{j\in \zz d}\fy (\cdo -j) = 1.
$$
Since $\maclH$ is continuously embedded in $\mascD '(\rr d)$ we get
$$
|(f,\psi _0)|
\lesssim
C\nm f{\maclH}\sum _{|\alpha |\le N}\nm {\partial ^\alpha \psi _0}{L^\infty (Q_{d,1+\ep})},
$$
for some constants $C>0$ and $N\ge 0$ which are
independent of $\psi _0\in C_0^\infty (Q_{d,1+\ep})$ and
$f\in \maclH$.

\par

Hence, if $\psi \in C_0^\infty (\rr d)$ we get
\begin{align*}
|(f,\psi )| &= \left |\left (  f, \sum _{j\in \zz d} \psi \cdot \fy (\cdo -j) \right ) \right |
\\[1ex]
&\le
\sum _{j\in \zz d} |(f(\cdo +j),\psi (\cdo +j)\fy )|
\\[1ex]
&\le
C\sum _{j\in \zz d} \nm {f(\cdo +j)}{\maclH}\sum _{|\alpha |\le N}
\nm {\partial ^\alpha (\psi (\cdo +j)\fy)}{L^\infty (Q_{d,1+\ep})}
\\[1ex]
&=
C\nm {f}{\maclH} \sum _{j\in \zz d} \sum _{|\alpha |\le N}
\nm {\partial ^\alpha (\psi (\cdo +j)\fy)}{L^\infty (Q_{d,1+\ep})}
\\[1ex]
&\le
C\cdot C_\fy \nm {f}{\maclH} \nmm{\psi} ,
\end{align*}
for some semi-norm $\nmm \cdo$ in $\mascS (\rr d)$. Here $C_\fy$
only depends on $\fy$. Hence,
\begin{equation}\label{Eq:SemiNormQuasiEst}
|(f,\psi )| \le C\nm {f}{\maclH} \nmm{\psi}
\end{equation}
for some constant $C$ which is independent of $f\in \maclH$ and
$\psi \in C_0^\infty (\rr d)$. Since $C_0^\infty (\rr d)$ is a dense subspace of
$\mascS (\rr d)$, it follows that the definition of $(f,\psi )$ extends uniquely
to any $f\in \maclH$ and $\psi \in \mascS (\rr d)$, and that
\eqref{Eq:SemiNormQuasiEst} holds. This shows that $\maclH$ is continuously
embedded in $\mascS '(\rr d)$.

\medspace

By Feichtinger's minimization property it follows that for $\maclH$ in Theorem
\ref{Thm:InvHilbertSpaces} we have
\begin{equation}\label{Eq:FirstInclusion}
\mascS (\rr d)\subseteq M^{1,1}(\rr d)\subseteq \maclH \subseteq M^{\infty ,\infty}(\rr d),
\end{equation}
with continuous inclusions. In particular, the
normalized standard Gaussian on $\rr d$, $h_0(x)= \pi ^{-\frac d4}e^{-\frac 12|x|^2}$ belongs
to $\maclH$.

\par

We have now the following lemma.

\par

\begin{lemma}\label{Lemma:GenPropInvHilbertSpaces}
Let $\maclH$ be a translation and modulation invariant Hilbert space
on $\rr d$. Then the following is true:
\begin{enumerate}
\item for every $f,g \in \maclH$ and $x,\xi \in \rr d$
it holds
\begin{align}
(f (\cdo -x),g )_{\maclH}
&=
(f ,g (\cdo +x))_{\maclH}
\label{Eq:GenPropInvHilbertSpaces1A}
\intertext{and}
(f \cdot e^{-i\scal \cdo \xi},g )_{\maclH}
&=
(f ,g \cdot e^{i\scal \cdo \xi})_{\maclH} \text ;
\label{Eq:GenPropInvHilbertSpaces1B}
\end{align}
%
%

\vrum

\item $(\partial ^\alpha f ,g )_{\maclH} = (f ,(-\partial )^\alpha g )_{\maclH}$
and $(x^\alpha f ,g )_{\maclH} = (f ,x^\alpha g )_{\maclH}$
for every $f ,g \in \mascS (\rr d)$ and $\alpha \in \nn d$.
\end{enumerate}
\end{lemma}

\par

\begin{proof}
We have
$$
\nm {f (\cdo -x)e^{-i\scal \cdo \xi}}{\maclH}^2 = \nm f{\maclH}^2,
$$
which by polarization gives
\begin{equation}\label{Eq:GenPropInvHilbertSpaces3}
(f (\cdo -x)e^{-i\scal \cdo \xi},g (\cdo -x)e^{-i\scal \cdo \xi})_{\maclH}
=
(f ,g )_{\maclH},
\end{equation}
when $f ,g \in \maclH$ and $x,\xi \in \rr d$. This gives (1).

\par

The assertion (2) follows by applying $\partial _x^\alpha$ and $\partial _\xi ^\alpha$ on
\eqref{Eq:GenPropInvHilbertSpaces1A} and \eqref{Eq:GenPropInvHilbertSpaces1B},
and then letting $x=\xi =0$.
\end{proof}

\par

We shall apply the previous result to deduce essential information of Hermite functions
and their role in the Hilbert space $\maclH$. We recall that the
Hermite function $h_\alpha$ of order $\alpha \in \nn d$ on $\rr d$
is defined by
$$
h_\alpha (x) = \pi ^{-\frac d4}(-1)^{|\alpha |}
(2^{|\alpha |}\alpha !)^{-\frac 12}e^{\frac 12{|x|^2}}
(\partial ^\alpha e^{-|x|^2}),\qquad x\in \rr d,\ \alpha \in \nn d.
$$
It is well-known that $\{h_\alpha \} _{\alpha \in \nn d}$ is an orthonormal basis for
$L^2(\rr d)$, and a basis for $\mascS (\rr d)$.

\par

We may pass between different Hermite functions by applying the
\emph{annihilation} and \emph{creation} operators, which are given by
$$
\operatorname{A}_j
= \frac 1{\sqrt 2}
\left (
x_j +\frac \partial {\partial x_j}
\right )
\quad \text{and}\quad
\operatorname{C}_j
= \frac 1{\sqrt 2}
\left (
x_j -\frac \partial {\partial x_j}
\right ),
$$
respectively, $j=1,\dots ,d$. It is then well-known that if $e_j$ is the $j$th
vector in the standard basis in $\rr d$, then
\begin{align}
\operatorname{A}_jh_\alpha
&=
\begin{cases}
\sqrt {\alpha _j}\, h_{\alpha -e_j}, & \alpha _j\ge 1,
\\[1ex]
0, & \alpha _j=0
\end{cases}
\label{Eq:HermiteAnn}
\intertext{and}
\operatorname{C}_jh_\alpha
&=
\sqrt {\alpha _j+1}\, h_{\alpha +e_j},
\quad \alpha \in \nn d.
\label{Eq:HermiteCreation}
\end{align}
(Cf. e.{\,}g. \cite{Ba1}.) This implies
\begin{multline}\label{Eq:SalphaTalphaHermite1}
A^\alpha h_\alpha = \alpha !^{\frac 12}\, h_0,
\quad \text{and}\quad
C^\alpha h_0 = \alpha !^{\frac 12}\, h_\alpha ,
\\[1ex]
\text{where}\quad
A^\alpha = \prod _{j=1}^d \operatorname{A}_j^{\alpha _j},
\quad \text{and}\quad 
C^\alpha = \prod _{j=1}^d \operatorname{C}_j^{\alpha _j}.
\end{multline}
Furthermore,
\begin{equation}\label{Eq:SalphaTalphaHermite2}
A^\beta h_\alpha =0
\quad \text{when}\quad
\alpha _j < \beta _j\quad 
\text{for some}\quad j\in \{ 1,\dots ,d\}.
\end{equation}

\par

\begin{proof}[Proof of Theorem \ref{Thm:InvHilbertSpaces}]
Suppose that $\alpha ,\beta \in \nn d$ are such that $\beta _j>\alpha _j$ for some
$j\in \{ 1,\dots ,d \}$. Then Lemma \ref{Lemma:GenPropInvHilbertSpaces} (2),
\eqref{Eq:SalphaTalphaHermite1} and \eqref{Eq:SalphaTalphaHermite2}
imply
$$
(h_\alpha ,h_\beta )_{\maclH} =
\beta !^{-\frac 12} (h_\alpha  ,C^\beta h_0)_{\maclH} =
\beta !^{-\frac 12} (A^\beta h_\alpha  ,h_0 )_{\maclH} =0
$$
and
\begin{align*}
\nm {h_\alpha}{\maclH}^2
&= (h_\alpha ,h_\alpha )_{\maclH} =
\alpha !^{-\frac 12} (h_\alpha  ,C^\alpha h_0)_{\maclH}
\\[1ex]
&=
\alpha !^{-\frac 12} (A^\alpha h_\alpha  ,h_0)_{\maclH}
=
(h_0 ,h_0)_{\maclH} = \nm {h_0}{\maclH}^2.
\end{align*}
This implies that
$
\{ \nm {h_0}{\maclH}^{-1}h_\alpha \} _{\alpha \in \nn d}
$
is an orthonormal system for $\maclH$.

\par

Hence, if $f\in \mascS (\rr d)$, then
\begin{multline*}
\nm f{\maclH}^2 = \Nm {\sum _{\alpha \in \nn d}(f,h_\alpha )_{L^2}h_\alpha }{\maclH}^2
\\[1ex]
=
\nm {h_0}{\maclH}^2\sum _{\alpha \in \nn d}|(f,h_\alpha )_{L^2}|^2
=
\nm {h_0}{\maclH}^2 \nm f{L^2}^2.
\end{multline*}
Since $\mascS (\rr d)$ is dense in $L^2(\rr d)$, it follows that $L^2(\rr d)$ is continuously
embedded in $\maclH$, and that \eqref{Eq:NormEquiv} holds. Furthermore, let from now on
the original $\maclH$ norm be replaced by
$f\mapsto \nm {h_0}{\maclH}^{-1}\nm f{\maclH}$. Then it follows that the inclusion
$i\, :\, L^2(\rr d)\to \maclH$ is an isometric injection.

\par

We shall use Hahn-Banach's theorem to prove that the latter map is in fact bijective.
Suppose that $\ell$ is a linear continuous form on $\maclH$ which is zero on $L^2(\rr d)$.
Then $\ell (f)=(f,g_0)_{\maclH}$ for some unique $g_0\in \maclH$. We need to prove that
$g_0=0$. Recall that the restriction of $(\cdo ,\cdo )_{L^2}$ on
$\mascS (\rr d)\times \mascS (\rr d)$ is uniquely
extendable to a continuous sesqui-linear form on $\mascS (\rr d)\times \mascS '(\rr d)$
and that the dual of $\mascS (\rr d)$ can be identified by $\mascS '(\rr d)$ through
this extension.
Since the forms $(\cdo ,\cdo )_{L^2}$ and $(\cdo ,\cdo )_{\maclH}$ agree on $L^2(\rr d)$
which contains $\mascS (\rr d)$ and that $\maclH \subseteq \mascS '(\rr d)$,
the same extension and duality properties hold true with $(\cdo ,\cdo )_{\maclH}$
in place of $(\cdo ,\cdo )_{L^2}$. In particular,
$$
(f,g)_{L^2} = (f,g)_{\maclH},\qquad f\in \mascS (\rr d),\ g\in \mascS '(\rr d).
$$

\par

We have $g_0\in \maclH \subseteq \mascS '(\rr d)$, and since any element
in $\mascS '(\rr d)$ has a Hermite series expansion, converging in
$\mascS '(\rr d)$, whose coefficients are polynomially bounded with respect
to their orders, it follows that
$$
g_0 = \sum _{\alpha \in \nn d} c(\alpha )h_\alpha ,
$$
for some $\{ c(\alpha )\} _{\alpha \in \nn d}$ such that $|c(\alpha )|\lesssim (1+|\alpha |)^N$
for some $N\ge 0$. Since $(f,g_0)_{\maclH}=0$ when $f\in L^2(\rr d)$, and that
$h_\alpha \in \mascS (\rr d)$ we get
$$
c(\alpha ) = \overline {(h_\alpha ,g_0)_{L^2}} = \overline {(h_\alpha ,g_0)_{\maclH}}=0,
$$
giving that $g_0=0$.

\par

By Hahn-Banach's theorem it follows that $L^2(\rr d)$ is dense in $\maclH$. Since
$L^2(\rr d)$ is also a closed subset of $\maclH$, it follows that $\maclH = L^2(\rr d)$,
and the result follows.
\end{proof}

\par

\end{document}